\documentclass[notitlepage,a4paper,aps,prd,onecolumn,superscriptaddress,nofootinbib,groupedaddress]{revtex4-2}
\usepackage{enumerate}
\usepackage{amsmath}
\usepackage{bm}
\usepackage{epsfig}
\usepackage{amsfonts}
\usepackage{amssymb}
\usepackage[utf8]{inputenc}
\usepackage[T1]{fontenc}
\usepackage{mathtools}
\usepackage[colorlinks=true]{hyperref}
\usepackage[dvipsnames]{xcolor}
\usepackage{graphicx}
\usepackage[normalem]{ulem}
\usepackage{enumerate}

\newcommand{\be}{\begin{equation}}
\newcommand{\ee}{\end{equation}}
\newcommand{\bea}{\begin{eqnarray}}
\newcommand{\eea}{\end{eqnarray}}
\newtheorem{theorem}{Theorem}

\newtheorem{conclusion}[theorem]{Conclusion}

\newtheorem{corollary}[theorem]{Corollary}

\newtheorem{lemma}[theorem]{Lemma}

\newenvironment{proof}[1][Proof]{\noindent\textbf{#1.} }{\ \rule{0.5em}{0.5em}}

\begin{document}

\title{Birkhoff Theorem for Berwald Finsler spacetimes}

\author{Nicoleta Voicu}
\email{nico.voicu@unitbv.ro}
\affiliation{Faculty of Mathematics and Computer Science, Transilvania University, Iuliu Maniu Str. 50, 500091 Brasov, Romania}
\author{Christian Pfeifer}
\email{christian.pfeifer@zarm.uni-bremen.de}
\affiliation{ZARM, University of Bremen, 28359 Bremen, Germany}
\author{Samira Cheraghchi}
\email{samira.cheraghchi@unitbv.ro}
\affiliation{Faculty of Mathematics and Computer Science, Transilvania University, Iuliu Maniu Str. 50, 500091 Brasov, Romania}

\begin{abstract}
Finsler spacetime geometry is a canonical extension of Riemannian spacetime geometry. It is based on a general length measure for curves (which does not necessarily arise from a spacetime metric) and it is used as an effective description of spacetime in quantum gravity phenomenology as well as in extensions of general relativity aiming to provide a geometric explanation of dark energy. A particular interesting subclass of Finsler spacetimes are those of Berwald type, for which the geometry is defined in terms of a canonical affine connection that uniquely generalizes the Levi-Civita connection of a spacetime metric. In this sense, Berwald Finsler spacetimes are Finsler spacetimes closest to pseudo-Riemannian ones. We prove that all Ricci-flat, spatially spherically symmetric Berwald spacetime structures are either pseudo-Riemannian (Lorentzian), or flat. This insight enables us to generalize the Jebsen-Birkhoff theorem to Berwald spacetimes.
\end{abstract}

\maketitle

\section{$SO(3)$ symmetric Berwald  Finsler gravity}
\label{sec:SO3BerwFins}
Spherically symmetric spacetime geometry is used to describe physical systems like black holes,  neutron stars, or the solar system, in their non-rotating approximation. Understanding the gravitational field of these systems in spherical symmetry is the first step toward a more realistic description in terms of axially symmetric spacetime geometry, which takes rotational effects into account.

In general relativity, the only family of spherically symmetric vacuum solutions to the Einstein equation are the Schwarzschild solutions, which are parameterized by the Schwarzschild radius. This fact is ensured by the famous Birkhoff (or Jebsen-Birkhoff) Theorem \cite{VojeJohansen:2005nd,1923Birk}.

In this article, we consider Finsler spacetime geometry, which is a very promising candidate as an extended geometry of spacetime to improve the description of gravity beyond general relativity, \cite{Beem,Javaloyes:2018lex,Javaloyes:2022hph,Carvalho:2022sdz, Fuster:2015tua, Heefer_2023, Li:2015sja, Minas:2019urp, Papagiannopoulos:2017whb, Pfeifer:2019, Saridakis:2021vue, Stavrinos2014,Tavakol2009, Triantafyllopoulos:2018bli}. In particular, Finsler geometry is employed to describe the gravitational field of kinetic gases on the tangent bundle without velocity averaging~\cite{Hohmann:2020yia}, a very possible source of dark energy, and finds applications as an effective description of a quantum spacetime in quantum gravity phenomenology including deformations of local Lorentz invariance \cite{Addazi:2021xuf, Gibbons:2007iu, Kostelecky:2011qz, Lobo:2020qoa, Mavromatos:2010nk, Zhu:2022blp}.

In Finsler spacetime geometry, the fundamental geometric object is, physically speaking, a point particle Lagrangian (technically a scalar-valued function) $L=L(x,\dot{x})$, depending on the spacetime points and velocity (tangent) vectors, and giving rise to a notion of arc length. This is a straightforward generalization of Riemannian spacetime geometry - which is recovered in the particular case when $L=g_{ab}(x)\dot{x}^a\dot{x}^b$ is quadratic in the velocity components. S.S. Chern coined the following sentence to summarize Finsler geometry: "Finsler (spacetime) geometry is (pseudo)-Riemannian geometry without the quadratic restriction" \cite{Chern1996}. Originally, he referred to positive definite Finsler geometry, but the same phrase easily carries over to indefinite Finsler geometry (and in particular, to Finsler spacetime geometry).\\
The Finslerian arc length - which is, in a spacetime setting, physically interpreted as proper time - allows one to define geodesics and curvature in a largely similar way to Riemannian geometry. In particular, it uniquely and canonically determines a generalization of the Levi-Civita connection - which is, though, generally not an affine connection on spacetime, but a \textit{nonlinear} connection \cite{Bucataru_2011,Bao}.\\
Among all Finsler spacetimes, the ones closest to pseudo-Riemannian geometry are \textit{Berwald spacetimes}, \cite{Berwald1926}. These are a distinguished subclass, singled out by the existence of a unique \textit{affine (or linear)} connection on spacetime with similar properties as the Levi-Civita one in Riemannian geometry.

\bigskip

In this context, we search for spherically symmetric solutions of the canonical action-based Finsler gravity equations \cite{Hohmann:2021zbt} and find that all $SO(3)$-symmetric, Finsler-Ricci flat Berwald spacetime solutions of the vacuum Finsler gravity equation must be either flat, or pseudo-Riemannian; actually, such a result even holds independenly of the signature, that is, for the more general class of 4-dimensional pseudo-Finsler spaces. This finding can be summarized as an extension of the Jebsen-Birkhoff Theorem to Berwald spacetime geometry: the only Ricci flat Berwald spacetimes with non-vanishing curvature are the usual pseudo-Riemannian, static and asymptotically flat Schwarzschild spacetimes.

In the case of small (linearized) Finslerian departures from the Schwarzschild metric, a similar result was proven by Rutz, back in 1998 \cite{RUTZ1998300}, which we know extend to general Berwald spacetimes.

The notation and conventions throughout this article are as follows. We consider a 4-dimensional manifold $M$, equipped with local coordinates $(x^a)$ and its tangent bundle $TM$, with naturally induced local coordinates $(x^a,\dot x^a)$; the local coordinate bases of the tangent spaces to $TM$ will be denoted by $\partial_a = \partial/\partial x^a$ and $\dot \partial_a = \partial/\partial\dot x^a$. Indices $a,b,c$ run from 0 to 3; whenever the local chart is fixed, we will omit for simplicity the indices of the coordinates of points $(x,\dot{x}) \in TM$.

\subsection{Finsler geometry in spherical symmetry}
In this section we briefly recall the basic mathematical notions of Finsler spacetimes and their dynamics.\vspace{6pt}

Roughly speaking, passing from Riemannian geometry to Finslerian one amounts to replacing the scalar product in each tangent space with just a \textit{norm} (that does not necessarily arise as the square root of some scalar product). In other words, in a Finsler space, the geometry is derived from a general length measure for curves
\begin{align}\label{eq:flength}
	S[\gamma] = \int F(\gamma, \dot\gamma) d\tau\,,
\end{align}
where the scalar-valued \textit{Finsler function}, or \textit{Finsler norm} $F=F(x,\dot x)$ is required to be positively $1$-homogeneous in~$\dot x$, to ensure the reparametrization invariance of $S[\gamma]$. All necessary geometric objects are then derived from $F$, or, more conveniently, from 
\begin{equation}
L:=F^2
\end{equation} 
and its derivatives. All details about positive definite Finsler geometry can be found for example in the book \cite{Bao}, while an extended discussion about Finsler spacetimes and the subtleties of extending Finsler geometry from positive definite to indefinite length measures, can be found, e.g, in the articles \cite{Hohmann:2021zbt,Javaloyes:2018lex}; below, we just briefly review the main notions to be used in the next sections.  

\vspace{6pt}

\noindent A \emph{pseudo-Finsler space}, \cite{Bejancu}, is a smooth manifold $M$ equipped with a function $L:\mathcal{A}\to \mathbb{R},\ (x,\dot x) \mapsto L(x,\dot x)$, where $\mathcal{A}\subset TM$ is a conic subbundle\footnote{A conic subbundle of the tangent bundle $TM$ is a subset of $TM\setminus\{0\}$ with non-empty fibers $\mathcal{A}_x=\mathcal{A}\cap T_{x}M\setminus\{0\}$, which is invariant under positive rescaling of vectors, i.e., if $(x,\dot{x}) \in \mathcal{A}$, then $(x,\lambda\dot{x}) \in \mathcal{A}$, for all $\lambda >0$.}, $L$ is 2-homogeneous in $\dot x$ and the Hessian of $L$ w.r.t. $\dot x$, called the Finsler metric tensor,
\begin{align}
	g_{ab}(x,\dot x)  = \frac{1}{2}\dot{\partial}_a\dot{\partial}_b L\,
\end{align}
is non-degenerate on $\mathcal{A}$. \vspace{6pt}

\noindent In particular, a \emph{Finsler spacetime} is a pseudo-Finsler space with a well defined physical causal structure. We demand that there exists a conic subbundle $\mathcal{T}\subset \mathcal{A}$, where the Finsler metric has Lorentzian signature $(+,-,-,-)$ and $L>0$ - and whose fibers $\mathcal{T}_x = \mathcal{T} \cap T_xM$ are physically interpreted as the cones of future pointing timelike directions $\mathcal{T}_x$ in each tangent space $T_xM$. Moreover, in order to ensure a compatible motion of massless particles (on the set where $L=0$) and massive particles (on the set $\mathcal{T}$), the boundaries $\partial\mathcal{T}_x$ shall be light cones, and $L$ must thus be continuously extended as $L=0$ to  $\partial\mathcal{T}_x$, \cite{Hohmann:2021zbt}. If this extension is smooth, then the Finsler spacetime is called \emph{proper}~\cite{Javaloyes:2021jqw}.  We mention that in the literature there exist several more suggestions how to define Finsler spacetimes, which differ in the mathematically precise smoothness conditions on $L$, but agree all on the requirement that a well-defined causal structure must exist \cite{Beem,Minguzzi:2014fxa,Caponio:2015hca}.

On Finsler spacetimes, the integral length measure \eqref{eq:flength} is physically interpreted as a geometric clock and as a point particle action. It measures the length of worldlines of particles and observers and determines the point particle equations of motion as being the Finsler geodesic equation.

The geometry of general Finsler spacetimes is formulated on the tangent bundle. The basic geometrical object is the unique, torsion free, Finsler metric compatible \textit{Cartan (canonical) nonlinear connection} on $TM$ that is 
the straightforward generalization of the  Levi-Civita connection to the Finsler geometry setting. It is uniquely determined by the Finsler function, just as the Levi-Civita connection is uniquely determined by a Riemannian metric.
For gravitational dynamics of the Finsler spacetimes, the curvature of the canonical nonlinear connection becomes the main object. Apart from this, we also need several non-Riemannian quantities, which we briefly introduce below.
\vspace{6pt}

As most of the results below hold true in arbitrary signature, unless elsewhere specified, we will work in the more general class of pseudo-Finsler spaces $(M,L)$.  \vspace{6pt}

The local coordinate expressions of the geometric objects (defined on $\mathcal{A}$) which we need in this article are:
\begin{itemize}
	\item The Cartan tensor $C_{abc}$ and its trace $C_a$. Cartan tensor is a measure of the deviation of a given Finsler geometry from (pseudo)-Riemannian geometry and vanishes if and only if $L$ is quadratic in $\dot x$ - i.e., if $L$ defines a pseudo-Riemannian geometry and the Finsler metric is independent of $\dot x$:
	\begin{align}
		C_{abc} = \frac{1}{4}\dot{\partial}_a\dot{\partial}_b \dot{\partial}_c L = \frac{1}{2}\dot{\partial_a} g_{bc}\,,\quad C_a = g^{bc}C_{abc}\,.
	\end{align}
	The trace of the Cartan tensor thus has the interpretation of an average of these deviations along various directions - and can be nicely expressed in terms of $\det(g)$:
	\begin{align}
	C_a = \dot{\partial}_a \ln(\sqrt{|\det(g)|})\,.
	\end{align} 
	\item Considering the arc length functional and applying the Euler-Lagrange equations, the geodesic equation in arclength parameterization is expressed in terms of the geodesic spray coefficients $G^a$:
	\begin{align}
		\ddot x^a + 2 G^a(x,\dot x) = 0\,,\quad 	G^a = \frac{1}{4}g^{ab}(\dot x^c\dot \partial_b\partial_c L- \partial_b L) \,.
	\end{align}	
	{The latter define} the canonical (Cartan) nonlinear connection coefficients $N^a{}_b$, which in turn define the horizontal derivative operators $\delta_a$, 
	\begin{align}
	N^a{}_b = \dot{\partial}_b G^a\,,\quad \delta_a = \partial_a - N^b{}_a \dot \partial_b\,.
	\end{align}
	We note that in general the nonlinear connection coefficients are not linear in $\dot x^a$, but only $1$-homogeneous, hence their name. They define a torsion-free connection, in the sense that, by construction, $\dot \partial_a N^c{}_b = \dot \partial_b N^c{}_a$.

	The horizontal derivative operators $\delta_a$ have the property that they transform under manifold induced coordinate transformations $(x, \dot x)\to(\tilde x(x), \dot {\tilde x}(x,\dot x)) = (\tilde x(x), \dot x^c  \partial_c \tilde x(x)^a)$ on $TM$, in the same way as the coordinate basis $\partial_a$ on spacetime. This way, the Cartan nonlinear connection coefficients ensure that Finsler spacetime geometric objects, which live on the tangent bundle, behave well under natural coordinate changes on $TM$.
	
	\item The curvature $R^a{}_{bc}$ of the canonical nonlinear connection and the (2-homogeneous) Finsler-Ricci scalar $R$ are given by
	\begin{align}
		R^a{}_{bc} = \delta_c N^a_b - \delta_b N^a_c, \quad R = R^a{}_{ab}\dot x^b\,.
	\end{align} 
	The former is the generalization of the Riemann curvature tensor and sources the geodesic deviation equation, i.e. gravitational tidal forces, while the latter is the canonical curvature scalar serving as a Lagrangian for the dynamics of Finsler spacetimes.
	
	\item The Landsberg tensor $S_{abc}$ and its trace $S_a$. The Landsberg tensor is obtained by applying the \textit{dynamical covariant derivative} $\nabla$ to the Cartan tensor; thus, it measures the rate at which the non-Riemannian character of the considered Finsler geometry varies as one moves along geodesics:
	\begin{align}
		S_{abc} = \nabla C_{abc} := \dot x^d \delta_d C_{abc} - N^d{}_a C_{dbc} -  N^d{}_b C_{adc} - N^d{}_c C_{abd}\,,\quad S_{a} = \nabla C_a =\dot x^b \delta_b C_a - N^b{}_a C_b\,.
	\end{align}
	Since, by construction of the Cartan nonlinear connection, the dynamical covariant derivative defined by coefficients $N^a{}_b$ is Finsler metric compatible, $\nabla g_{ab}=0$, we have that taking the trace with the Finslerian metric and covariant differentiation commute.
\end{itemize}
Before we use these notions to defined Berwald spacetimes in the next subsection we introduce the notion of spherical symmetry.

\emph{Spatially spherically symmetric} pseudo-Finsler spaces are defined by the existence of three Killing vector fields\footnote{Killing vector fields are vector fields on spacetime along whose flow lines the Finsler function is constant.} $X_i$ of $(M,F)$, which generate the $\mathfrak{so}(3)$ Lie algebra. 
In the following, let us fix a local chart on $M$ and work in spherical coordinates $(x^a):=(t,r, \theta,\phi)$ on the respective chart. Then, the canonical lifts to $TM$ of these generators are known to be, in naturally induced coordinates: $(t,r,\theta,\phi,\dot{t},\dot{r},\dot{\theta},\dot{\phi})$:
\begin{align}
	\begin{split}
		&X^C_1=\sin\phi\partial_\theta+\cot\theta\cos\phi\partial_\phi+\dot{\phi}
		\cos\phi\dot{\partial}_\theta-\bigg(\dot\theta\frac{\cos\phi}{\sin^2\theta}
		+\dot\phi\cot\theta\sin\phi\bigg)\dot{\partial}_\phi\,, \\
		&X^C_2=-\cos\phi\partial_\theta+\cot\theta\sin\phi\partial_\phi+\dot{\phi}
		\sin\phi\dot{\partial}_\theta-\bigg(\dot\theta\frac{\sin\phi}{\sin^2\theta}
		+\dot\phi\cot\theta\cos\phi\bigg)\dot{\partial}_\phi\,, \\
		&X^C_3=\partial_\phi\,.
	\end{split}
\end{align}
 Solving the Finslerian Killing equation $X^C(L)=0$ for these vector fields, results in the fact, \cite{Pfeifer:2013gha}, that $L$ can only depend on the coordinates $(t,r,\theta,\phi,\dot{t},\dot{r},\dot{\theta},\dot{\phi})$ in a very constrained way, namely,
\begin{align}
	L(x,\dot x) = L(t,r,\dot t, \dot r, w),\quad w^2 = \dot	\theta^2 +  \dot \phi^2 \sin^2\theta\,.
\end{align}

\subsection{Berwald geometry in spherical symmetry}
A pseudo-Finsler space is said to be \textit{of Berwald type } if the geodesic spray coefficients $G^a$ are (in one, then, in any local chart) quadratic in the coordinates $\dot{x}^a$, accordingly, the nonlinear connection coefficients are linear in $\dot{x}^a$:
\begin{align}
	 G^a  = \frac{1}{2}\Gamma^a{}_{bc}(x) \dot x^b \dot x^c\,,\quad N^a{}_b = \Gamma^a{}_{bc}(x) \dot x^c\,.
\end{align}
This way, Berwald spaces are pseudo-Finsler spaces which are closest to pseudo-Riemannian geometry, in the sense that they still admit a unique torsion-free affine connection (with coefficients $\Gamma^a{}_{bc}(x)$) on the base (spacetime) manifold $M$, which is compatible with the Finsler metric. In Lorentzian signature, this connection is yet not necessarily the Levi-Civita connection of any pseudo-Riemannian metric, see \cite{Fuster:2020upk} - making the geodesic structure of Berwald spacetimes more interesting and nontrivial.

Coming back to Berwald spaces of arbitrary signature, the (2-homogeneous) Finsler-Ricci scalar $R=R^{a}{}_{ab}\dot{x}^b$, see the Appendix \ref{app:defs} below, is also quadratic in the velocities; moreover, the trace of the Landsberg tensor vanishes:
\begin{align}
	R = R_{ab}(x)\dot x^a \dot x^b\,, \quad S_a = 0\,.
\end{align}

In a previous paper, \cite{Cheraghchi:2022zgv}, where all details can be found, we showed that nontrivially Finslerian, 4-dimensional $SO(3)$-symmetric Berwald metrics exist and can be divided into five classes. Below, we briefly review the technical details on these metrics and on their respective canonical connections which we will need in the following.

Let $L$ be a generic $SO(3)$-invariant, non-Riemannian Berwald function and denote by $\Gamma$ its canonical connection. Then, the coefficients $\Gamma^{a}_{bc}$ are given by 10 free functions $k_i=k_i(t,r)$ and the curvature tensor of $\Gamma$ is described by 14 coefficients $a_i=a_i(t,r)$ (see Appendix \ref{app:defs} for the precise formulas of $k_i$ and $a_i$).
Moreover, we assume for now that $k_7,k_8,k_9,k_{10}$ are not all zero (the case when $k_7=k_8=k_9=k_{10}=0$ will be treated separately), in particular we assume that $k_{10}\neq0$, so that we can define
\begin{equation}
a = \frac{k_7}{k_{10}}, \quad b = \frac{k_8}{k_{10}}, \quad c = \frac{k_9 k_{10} - k_7 k_8}{k_{10}^2}.
\end{equation} 
The above quantities allow us to introduce several relations:
\begin{itemize}
	\item The curvature components $a_6,...,a_{13}$ are subject to the constraints:
	\begin{equation}
	a_{6}=aa_{7},~a_{8}=ba_{7},~\ a_{9}=\left( ab+c\right) a_{7},~\
	a_{10}=aa_{11},~a_{12}=ba_{11},~\ \ a_{13}=\left( ab+c\right) a_{11}\,.
	\label{metrizability_conds2}
	\end{equation}
	\item There hold the identities:
	\begin{equation}
	a_{5}=a_{9}-a_{12}=\left( ab+c\right) a_{7}-ba_{11}\,.  \label{a5_identity}
	\end{equation}
	\item The curvature components $a_1,...,a_5$ are always related as:, 
	\begin{align}\label{A-F_defdef}
	\begin{split}
	A &:=b\left( aa_{1}+a_{2}\right) +\left( ab+c\right) \left(aa_{3}+a_{4}\right) -a_{5}\left( 2ab+c\right)=0\,,   \\
	B &:=a\left( aa_{3}+a_{4}\right) -\left( aa_{1}+a_{2}\right)=0\,, \\
	C &:=\left( ab+c\right) a_{3}+b\left( aa_{3}+a_{4}\right) +b (a_{1}-2a_{5})=0\,.
	\end{split}
	\end{align}
	\item Moreover, the quantities:
	\begin{align} \label{def_DEF}
		D := aa_{3}-a_{1}+a_{5}\,,\quad E:=ba_{3}\,,\quad F:=aa_{3}-a_{1}\,,
	\end{align}
	will help us discriminate among the possible classes of Berwald-Finsler metrics $L$.
\end{itemize}
With the above notations, the five classes of nontrivially Finslerian, $SO(3)$-invariant 4-dimensional Berwald metrics found in \cite{Cheraghchi:2022zgv} split into two groups I and II:
\begin{itemize}
	\item[I.] For connections with $k_7, k_8, k_9, k_{10}$ not all zero:
	\begin{enumerate}
	\item Power law type $D \neq 0$,
	\begin{align}\label{eq:1-L2}
		L = \vartheta(t,r) u^{2-2 \lambda}\left(v + \rho u^{2}\right)^\lambda\,, 
	\end{align}
	where $\lambda=\frac{F}{D}$ is a constant, $\rho=\frac{E}{D}$, $u = \dot t - a\dot r$, $v = c\dot{r}^2 + 2 b\dot t\dot r - w^2$ and $\vartheta=\vartheta(t,r)$ can be determined in terms of $k_i$ (see \cite{Cheraghchi:2022zgv} for the details).
	\item Exponential type $D=0, E\neq 0$,
	\begin{align}\label{eq:1-L3}
		L = \varphi(t,r) u^2 e^{\frac{v}{u^2}\mu}\,, 
	\end{align}
	where $\mu=\frac{F}{E}$, $u,v$ are as above and and $\varphi=\varphi(t,r)$ can be, again, determined in terms of $k_i$ (see \cite{Cheraghchi:2022zgv}).
	\item A function of one variable $D=E=F=0$,
	\begin{align}\label{eq:1-L5}
		L=u^2\Xi(z), \quad  z= \frac{v(\dot {\tilde  t}, \dot {\tilde r}, w)}{u(\dot {\tilde  t}, \dot {\tilde r}, w)^2}\,,
	\end{align}
	where $\Xi=\Xi(z)$ is an arbitrary function of the single variable $z$ and $u,v$ as defined above are, up to a suitable coordinate change, independent of  the new coordinates $\tilde t$ and $\tilde r$.
	\end{enumerate}
\item[II.] For connections with $k_7=k_8=k_9=k_{10}=0$:
\begin{enumerate}
	\item[4.] A function of another single variable,
	\begin{align}\label{eq:2-L2}
		L =  w^2 \xi(q)\,,\quad q = \frac{\dot t e^{I-\varphi}}{w}\,,
	\end{align}
	where $\xi=\xi(q)$ is an arbitrary function of the single variable $q$, $\varphi$ is a function of $t,r$ only and $I=I(t,r,p)$ is the integral of a certain rational function of $p=\frac{\dot{r}}{\dot{t}}$.
	\item[5.] Independent of $t$ and $r$ (up to a suitable coordinate change), function of two variables:
	\begin{align}\label{eq:2-L1}
		L = L(\dot t, \dot r, w) = \dot t^2 L(1, p, s)\,, \quad p = \frac{\dot r}{\dot t},\quad s = \frac{w}{\dot t}\,.
	\end{align}
\end{enumerate}
\end{itemize}
An immediate question is what kind of solutions of the canonical action-based Finsler gravity equations do exist in each of these classes - and if  a Birkhoff-type theorem holds.

\subsection{Berwald Finsler gravity}
One of the first Finsler generalizations of the vacuum Einstein equations, given by the vanishing of the Finsler-Ricci scalar $R=0$, had been proposed by Rutz \cite{Rutz}. It could however be shown that this equation cannot be obtained from the variation of an action functional. It turned out that the "closest" action-based vacuum Finsler gravity equation to $R=0$ (its so-called \textit{canonical variational completion}, see \cite{Hohmann:2018rpp}) is derived from the following action which can be formulated on the unit tangent (or indicatrix) bundle $I=\{(x,\dot x)\in\mathcal{A}\ |\ |L(x,\dot x)|=1\}$:
\begin{align}
	S_D[L] = \int_{D\subset I} \left( \frac{R}{L}\ d\Sigma \right)\,,
	\quad d\Sigma = \frac{|\det(g)|}{L^2} \mathbf{i}_{\dot x^a \dot \partial_a}(d^4x\wedge d^4\dot x) \,,
\end{align}
where $\mathbf{i}$ denotes the interior product between differential forms and vectors and the domain $D$ is compact. Technical details about the construction of the action and all mathematical subtleties can be found in \cite{Hohmann:2021zbt}.
Variation of the above action, supplemented by a source term given by the 1-particle distribution function $\phi$ of a kinetic gas, leads to:
\begin{align}
	g^{ab}\dot{\partial}_a\dot{\partial}_bR - 6 \frac{R}{L} + g^{ab}(\nabla_a S_b + S_a S_b +{\dot\partial}_a \nabla S_b)= \kappa \phi\,,
\end{align} 
where $\kappa$ is a constant and the symbols $\nabla_a$ and $\nabla$ denote the so-called horizontal Chern, respectively, the already encountered dynamical covariant derivative operators. As it will immediately turn out, these two operators are not needed for our further discussion, hence we do not discuss them here further. All details on the derivation of the Finsler gravity equation and its advantage over alternative suggestions for dynamics for Finsler spacetimes  can be found in \cite{Hohmann:2019sni,Hohmann:2021zbt}, and references therein.

This general equation is very involved  and difficult to solve. For Berwald spaces however, it simplifies tremendously, due to the vanishing of the mean Landsberg tensor $S_a$ and the simple $\dot x$-dependence of the Finsler Ricci scalar. The  vacuum field equation ($\phi=0$) reduces, in this case, to
\begin{align} \label{eq:vacuum_Berwald}
	g^{ab}(x,\dot x)\dot{\partial}_a\dot{\partial}_bR - 6 \frac{R}{L} = 0\,.
\end{align}
We call a pseudo-Finsler space \textit{Ricci-flat}, if its Finsler-Ricci scalar $R$ identically vanishes:
\begin{equation} \label{eq:R=0}
R=0.
\end{equation}
Obviously, \eqref{eq:R=0} is a sufficient condition for the vacuum equation \eqref{eq:vacuum_Berwald} to hold. Moreover, for \textit{proper Berwald spacetimes}, which are Berwald spacetimes such that $L$ extends smoothly from the cone of unit timelike directions to the null structure $L=0$, it has been proven, \cite{Javaloyes:2021jqw}, that this condition is also necessary, i.e., Finsler-Ricci flatness is actually \textit{equivalent} to the vacuum equation~\eqref{eq:vacuum_Berwald}.

Since in Berwald spaces the Finsler-Ricci scalar is quadratic in the velocities $R=R_{ab}(x)\dot{x}^a\dot{x}^b$ (where $R_{ab}$ denote the Ricci tensor components of its \textit{connection}) $R=0$ is equivalent to the vanishing of the symmetric part\footnote{For non-Riemannian Berwald spaces, the connection Ricci tensor $R_{ab}$ is not necessarily symmetric, see \cite{Fuster:2020upk}} of $R_{ab}$, that is, 
\begin{align}
\dot{\partial}_a\dot{\partial}_bR \equiv R_{ab}+R_{ba}=0.
\end{align}
	
It is still work in progress to answer the question if the condition $R_{ab}(x) = 0$ is necessary to solve the Finsler gravity equation for Berwald spacetimes, if one relaxes the smoothness request for $L$ on the light cone at each point $x \in M$ and only requires continuity instead. 

\section{Ricci flat vacuum solutions}
\label{sec:SO3BerwRicFlat}
Here is the main result of this paper, which we will prove over the next subsections:

\begin{theorem}[A Berwald Jebsen-Birkhoff Theorem]\label{thm:JB}
Let $(M,L)$ be a four-dimensional spatially spherically symmetric Berwald pseudo-Finsler space which is Finsler-Ricci flat, i.e., $R = 0$. Then $(M,L)$ is either flat, $R^a{}_{bc}=0$, or $L = g_{ab}(x)\dot x^a \dot x^b$ is the Finsler function of a pseudo-Riemannian metric $g_{ab}=g_{ab}(x)$ that is Ricci-flat, $R_{ab}=0$.
\end{theorem}
In Lorentzian signature, the above result has two immediate, yet, important consequences.

\begin{corollary}
	The only non-flat, Finsler-Ricci flat and spatially spherically symmetric Berwald spacetime geometries are given by the Finsler function of the Schwarzschild metric.
\end{corollary}
The corollary follows trivially from Theorem \ref{thm:JB}, since non-flat but Finsler-Ricci flat Berwald spacetimes are pseudo-Riemannian, and for Ricci flat pseudo-Riemannian spacetimes, the classical Jebsen-Birkhoff theorem holds.\\

Moreover, for the class of \textit{proper} Finsler spacetimes, we can make an even stronger statement:
\begin{corollary}
If a proper, non-flat spatially spherically symmetric Berwald spacetime is a solution of the Finsler vacuum equation \eqref{eq:vacuum_Berwald}, then it is the Schwarzschild spacetime.
\end{corollary}

\vspace{6pt}

To prove Theorem \ref{thm:JB}, we rely on the local classification of spherically symmetric Berwald spacetimes in \cite{Cheraghchi:2022zgv}, which we have briefly presented above. The strategy is the following: after proving two preliminary lemmas, we will show, for each of the possible classes of non-Riemannian $SO(3)$-symmetric Berwald structures separately, that the equation $R=0$ has either no solutions at all, or leads to the vanishing of the curvature tensor components $R^a{}_{bcd}=0$ of the spherically symmetric affine connection.

\subsection{Preliminary lemmas}
Here are two lemmas underlying our proof. 
A first step is to rewrite the Ricci flatness condition in a convenient way, as follows.

\begin{lemma} 
	\label{lem:Ric}
An $SO(3)$-symmetric 4-dimensional Berwald pseudo-Finsler space is Finsler-Ricci flat if and only if its canonical connection satisfies:
\begin{equation}
\left\{ 
\begin{array}{l}
a_{3}=-2ba_{7} \\ 
a_{1}-a_{4}-2a_{5}= 4ba_{11} \\ 
a_{2}=2\left( ab+c\right) a_{11} \\ 
aa_{7}+a_{11}=a_{14}
\end{array}
\right.   \label{eq:R=0_system}
\end{equation}
\end{lemma}

\begin{proof}
A direct calculation using the expressions \eqref{eq:appcon} and \eqref{eq:appcurv} of the curvature
coefficients of spherically symmetric Berwald spaces, gives the Finsler-Ricci scalar $R$ as:
\begin{equation}
R=-\left( a_{3}+2a_{8}\right) \dot{t}^{2}+\left( a_{1}-a_{4}-2\left(
a_{9}+a_{12}\right) \right) \dot{t}\dot{r}+\left( a_{2}-2a_{13}\right) \dot{r}^{2}+\left( a_{6}+a_{11}-a_{14}\right) w^{2}.  \label{eq:R}
\end{equation}
Since $a_{i}$ have no $\dot{x}$-dependence, we find that the equation $R=0$ is equivalent to the vanishing of each of the coefficients of $\dot{t}^2, \dot{t}\dot{r}, \dot{r}^2$ and $w^2$.

Taking into account the pseudo-Finsler metrizability conditions \eqref{metrizability_conds2} together with the identity \eqref{a5_identity}, this is nothing but the system \eqref{eq:R=0_system}. 
\end{proof}

A second useful result refers to $a_{14}$, which is the only curvature coefficient that is algebraic in $k_i$.

\begin{lemma}
	\label{lem:a14} If $k_{10}\not=0,$ then, there hold the identities:
	\begin{itemize}
		\item[(i)] $a_{14}=1+k_{10}^{2}\left( 2ab+c\right);$
		\item[(ii)] $\partial _{t}a_{14}=-2k_{10}(2ab+c)a_{7};$ $\partial_{r}a_{14}=-2k_{10}(2ab+c)a_{11}.$
	\end{itemize}
\end{lemma}

\begin{proof}
	\begin{itemize}
		\item[(i)] follows immediately from the identity
		\begin{equation}
			a_{14}=1+k_{7}k_{8}+k_{9}k_{10}  \label{def:a14}
		\end{equation}
		(see Appendix \ref{app:defs}) and from the definitions of $a,b$ and $c$.
		\item[(ii)]Differentiating relation (\ref{def:a14}) with respect to $t,$
		we find:
		\begin{equation}
			\partial _{t}a_{14}=k_{8}\partial _{t}k_{7}+k_{7}\partial
			_{t}k_{8}+k_{10}\partial _{t}k_{9}+k_{9}\partial _{t}k_{10}=-\left(
			k_{8}a_{6}+k_{7}a_{8}+k_{10}a_{9}+k_{9}a_{7}\right) ,
		\end{equation}
		where to prove the second equality, we have used relations \eqref{a_i} in the
		Appendix. Then, taking $k_{10}$ as a common factor and using:
		\begin{equation}
			k_{8}=bk_{10},~\ k_{7}=ak_{10},~\ \ k_{9}=\left( ab+c\right) k_{10},
		\end{equation}
		together with the Finsler metrizability conditions \eqref{metrizability_conds2}, this becomes:
		\begin{equation}
			\partial _{t}a_{14}=-2k_{10}(2ab+c)a_{7}
		\end{equation}
		as claimed. The second identity of \textit{(ii)} follows in a completely
		similar manner.
	\end{itemize}
\end{proof}

\bigskip
We are now ready to prove Theorem \ref{thm:JB} - that is, to solve the system \eqref{eq:R=0_system} - case by case.

\subsection{Class 1:\ Power law metrics $D \neq 0$}

A first useful remark is that $b$ and $c$ cannot simultaneously
vanish; this is justified by direct computation of $\det(g)$ - which, in the case $b=c=0,$ identically vanishes  (see also \cite{Cheraghchi:2022zgv}). An immediate consequence is that $b$ and $ab+c$ cannot be zero at the same time.

Taking into account this remark, we first solve the constraints $A=B=C=0$ (see \eqref{A-F_defdef}). A direct, elementary computation leads to three possibilities:
\begin{eqnarray}
	\left( I\right)  &:&b\not=0,~2ab+c\not=0\Rightarrow ~a_{1}=a_{5}-\dfrac{ab+c}{b}a_{3},~ a_{2}= \dfrac{a(ab+c)a_3}{b} ~\ a_{4}=a_{5}-aa_{3},~\ a_{3},a_{5}\in \mathbb{R}; \label{eq:I}\\
	(II) &:&b\not=0,~2ab+c=0\Rightarrow
	~a_{1}=-a_{4}+2a_{5},~a_{2}=a^{2}a_{3}+2a\left( a_{4}-a_{5}\right)
	,~a_{3},a_{4},a_{5}\in \mathbb{R}; \label{eq:II}\\
	(III) &:&b=0,~a_{2}=-aa_{1}+aa_{5},~a_{3}=0,~a_{4}=a_{5},~\ a_{1},a_{5} \in \mathbb{R}. \label{eq:III}
\end{eqnarray}

Let us investigate separately each of these possibilities.

\subsubsection{Case (I):\ $b\not=0,$ $c+2ab\not=0$}
The defining relations \eqref{eq:I}, the identity $a_{5}=\left( ab+c\right) a_{7}-ba_{11}$ and the first
equation \eqref{eq:R=0_system} allow us to obtain $a_{1},a_{2},a_{4}$ in terms of $a_{7}$ and 
$a_{11}$ as:
\begin{equation}
a_{1}=3\left( c+ab\right) a_{7}-ba_{11},a_{2}=-2a\left( c+ab\right)
a_{7},a_{4}=\left( c+3ab\right) a_{7}-ba_{11},  \label{eq:a1a2a4_I}
\end{equation}
which, substituted into the second and into the third equations  \eqref{eq:R=0_system} give respectively:
\begin{equation}
b\left( aa_{7}+a_{11}\right) =0,~\ \ \left( ab+c\right) \left(
aa_{7}+a_{11}\right) =0.
\end{equation}
Since, as noted above, $b$ and $ab+c$ cannot simultaneously vanish, this
leads to: $aa_{7}+a_{11}=0,$ which, using now the fourth equation \eqref{eq:R=0_system}, tells
us that
\begin{equation}
a_{14}=0.
\end{equation}
But, the latter, taking into account Lemma \ref{lem:a14} \textit{(ii)} and $k_{10}\neq 0, 2ab+c\neq 0$, implies that $
a_{7}=a_{11}=0,$ which, together with (\ref{eq:a1a2a4_I}) gives: 
\begin{equation}
a_{1}=a_{2}=a_{3}=a_{4}=a_{5}=0;
\end{equation}
in particular, $D=aa_{3}-a_{1}+a_{5}=0,$ in contradiction with the hypothesis $D\not=0$ which must hold for all power law metrics.

Therefore, this subclass contains no Ricci-flat Finsler functions.

\subsubsection{Case (II):\ $c=-2ab,$ $b\not=0$}\label{ssec:II}
In this case, Lemma \ref{lem:a14} together with $2ab+c=0$ gives:
\begin{equation}
a_{14}=1,
\end{equation}
which, using the fourth equation \eqref{eq:R=0_system}, then tells us that:
\begin{equation}
aa_{7}+a_{11}=1
\end{equation}
and accordingly, 
\begin{equation}
a_{5}=-b\left( aa_{7}+a_{11}\right) =-b.
\end{equation}
The latter, together with the first equation of \eqref{eq:R=0_system} and \eqref{eq:II}, allows us to express $a_{1}$
and $a_{2}$ as: 
\begin{equation}
a_{1}=-a_{4}-2b,~\ \ a_{2}=2a (b a_{11}+a_{4}).
\end{equation}
The remaining equations of \eqref{eq:R=0_system} then both give:
\begin{equation}
a_{4}=-2ba_{11}
\end{equation}
Returning to $a_{1},$ we find: $a_{1}=2b\left( a_{11}-1\right) =-2aba_{7}$
and finally,
\begin{equation}
F=aa_{3}-a_{1}=-2aba_{7}+2aba_{7}=0.
\end{equation}
But, the latter leads to $\lambda =\dfrac{F}{D}=0$. The latter implies that $L$ only depends on $u$, hence only on $\dot t$ and $\dot r$ (and is completely independent of $\dot \theta$ and $\dot \phi$), meaning that $L$ is actually degenerate.

We conclude that neither this subclass contains any nondegenerate solution of the equation $R=0$.

\subsubsection{Case (III): $b=0,$ $2ab+c\not=0$}

In this case, the identity $a_{5}=\left( ab+c\right) a_{7}-ba_{11}$ together
with the defining relations \eqref{eq:III} give:
\begin{equation}
a_{4}=a_{5}=ca_{7},~\ \ a_{3}=0,~\ a_{2}=a\left( -a_{1}+ca_{7}\right) .
\end{equation}
The first equation \eqref{eq:R=0_system} is thus trivially satisfied; the second one gives $
a_{1}=3ca_{7}$, which, substituted into the third one then reveals that
\begin{equation}
-2c\left( a_{11}+aa_{7}\right) =0.
\end{equation}
Since $b$ and $c$ cannot simultaneously vanish, this implies:
\begin{equation}
a_{14}=a_{11}+aa_{7}=0
\end{equation}
and therefore, by Lemma \ref{lem:a14}, we find $a_{7}=a_{11}=0;$ this leads
again, to 
\begin{equation}
a_{1}=a_{2}=...=a_{5}=0,
\end{equation}
in contradiction with $D\not=0$. Therefore, this subclass does not contain any Finsler-Ricci flat structures either.

Summing up, we are led to the following conclusion:

\begin{conclusion}
	There is no nontrivially Finslerian, power law Berwald metric $L$ obeying $R=0.$
\end{conclusion}

\subsection{Class 2:\ exponential type metrics $D=0, E\neq 0$}

A first remark:\ in this case, $E=ba_{3}\not=0,$ which means that both $b$ and $a_{3}$ must be nonzero. Hence, we must solve the system $A=B=C=D=0,$ under this condition. The only possible solution then is
\begin{equation}
2ab+c=0,~\ a_{1}=aa_{3}+a_{5},~\ a_{2}=-a^{2}a_{3},a_{4}=-aa_{3}+a_{5}\,.
\end{equation}
In a completely similar manner to Class 1 case (II) (see paragraph \ref{ssec:II} above), we find from Lemma \ref{lem:a14}, the equality $2ab+c=0$ and the fourth equation \eqref{eq:R=0_system}, that
\begin{equation}
a_{14}=aa_{7}+a_{11}=1,~\ \ a_{5}=-b\,.
\end{equation}
The first equation \eqref{eq:R=0_system} says that $a_{3}=-2ba_{7}$, and thus
\begin{equation}
a_{1}=-b\left( 2aa_{7}+1\right) ,~\ a_{2}=2a^{2}ba_{7},~a_{4}=b\left(
2aa_{7}-1\right) .
\end{equation}
The second equation \eqref{eq:R=0_system} now reads:
\begin{equation}
-2b\left( 2a_{11}+2aa_{7}-1\right) =0,
\end{equation}
which, taking into account that $aa_{7}+a_{11}=1,$ implies:
\begin{equation}
b=0,
\end{equation}
that is, $E=ba_{3}=0$, which is in contradiction with our hypothesis, that is, the system is inconsistent.

\begin{conclusion}
	There is no exponential-type Berwald, spatially spherically
	symmetric metric obeying $R=0.$
\end{conclusion}

\subsection{Class 3: $D=E=F=0$}
For this case, see again \cite{Cheraghchi:2022zgv}, the condition $\left[ \delta_{t},\delta _{r}\right] =0$ holds and  is equivalent to:
\begin{equation*}
	a_{1}=a_{2}=...=a_{5}=0.
\end{equation*}
Thus, the equation $R=0$ (equivalently, the system (\ref{eq:R=0_system}))\
reads
\begin{equation}
	ba_{7}=0,~\ \ ba_{11}=0,~\ \left( ab+c\right) a_{11}=0,~\ \ a_{14}=aa_{7}+a_{11}.
	\label{a7a11_flat_case}
\end{equation}
Since, as already mentioned above, $b$ and $c$ cannot
simultaneously vanish (as this would lead to a degenerate metric), we must always have $a_{11}=0,$ which, combined with (\ref{a5_identity}) then gives:
\begin{equation*}
	\left( ab+c\right) a_{7}=0;
\end{equation*}
this, together with the first equation (\ref{a7a11_flat_case}) tells us that $a_{7}=0$ and therefore: 
\begin{equation*}
	a_{1}=a_{2}=...=a_{13}=a_{14}=0,
\end{equation*}
in other words the canonical connection of $L$ must be flat.

\begin{conclusion}
	Any Ricci-flat, Berwald spatially symmetric spacetime obeying $\left[ \delta
	_{t},\delta _{r}\right] =0$, is flat.
\end{conclusion}

\subsection{Classes 4 and 5:\  $k_{7}=k_{8}=k_{9}=k_{10}=0$}

For these metrics, the hypothesis $k_{7}=k_{8}=k_{9}=k_{10}=0,$ together with the expressions of the curvature components $a_i$ in Appendix \ref{app:defs}, give:
\begin{equation*}
	a_{5}=...=a_{13}=0\,,\quad a_{14}=1.
\end{equation*}
In particular, we obtain a contradiction to the fourth equation \eqref{eq:R=0_system}
\begin{equation*}
	aa_{7}+a_{11}-a_{14}=-1\not=0.
\end{equation*}
We thus obtain:
\begin{conclusion}
There are no Ricci-flat metrics belonging to Classes 4 or 5. 
\end{conclusion} 

The statement of Theorem \ref{thm:JB} then follows immediately from the Conclusions of all the above subsections.

\section{Conclusion}\label{sec:conc}
In our main result, Theorem \ref{thm:JB}, we have shown that for $SO(3)$-symmetric Berwald spacetimes (Finsler spacetimes closest to pseudo-Riemannian geometry) Finsler-Ricci flatness $R=R_{ab}(x)\dot x^a \dot x^b = 0$ implies flatness or pseudo-Riemannian geometry. Since for the latter, the Jebsen-Birkhoff theorem holds, all Finsler-Ricci flat but non-flat Berwald spacetimes are given by the Finsler function of the Schwarzschild metric
\begin{align}
	L = -(1-r_s/r)\dot t^2 + (1-r_s/r)^{-1} \dot r^2 + r^2 (\dot \theta^2 + \sin^2\theta \dot \phi^2)\,.
\end{align}
With this finding, we extended the Jebsen-Birkhoff  theorem stating that spatial spherical symmetry together with Ricci flatness imply Schwarzschild geometry, to Berwald spacetimes. 

In the context of general relativity, Ricci flatness is equivalent to the vacuum gravitational field equations - the vacuum Einstein equations. In Finsler geometry, the gravitational field equation is still equivalent to Ricci flatness  for \textit{proper} Berwald-Finsler spacetimes, that are defined by the fact that the Finsler function is smooth on the light cone. In general in modified gravity theories and also for general Berwald-Finsler spacetimes, this relation between Ricci flatness  and the vacuum field equations does not necessarily hold.

Currently ongoing research explores the existence of spherically symmetric Berwald spacetime solutions to the Finsler gravity vacuum equation that  are non-Ricci flat. These solutions, if any, would be deviations from the pseudo-Riemmanian Finsler function induced by the Schwarzschild metric and necessarily only continuous (i.e., not smooth) on the light cone. In either case (existence or non-existence of non Ricci flat solutions) this will be an important insight into the classification of spherically symmetric solutions to the Finsler gravity equation.

\appendix

\section{Connection coefficients and curvature components}\label{app:defs}
For the proof of Theorem~\ref{thm:JB} we need the following notions introduced in \cite{Cheraghchi:2022zgv}.

The most general $SO(3)$-invariant, symmetric affine connection is given, in 4 dimensions, by the coefficients:
\begin{align}
	\Gamma_{tt}^{t} &=k_{1}(t,r), & 
	\Gamma_{tr}^{t} &=k_{2}(t,r), \nonumber\\
	\Gamma_{rr}^{t} &=k_{3}(t,r), &
	\Gamma_{tt}^{r} &=k_{4}(t,r), \nonumber\\
	\Gamma_{rr}^{r} &=k_{5}(t,r), & 
	\Gamma_{tr}^{r} &=k_{6}(t,r), \nonumber\\
	\Gamma_{\theta \theta }^{t} &=\tfrac{\Gamma _{\phi \phi }^{t}}{\sin ^{2}\theta }=k_{7}(t,r), & 
	\Gamma _{\phi t}^{\phi }    &=\Gamma _{\theta t}^{\theta}=k_{8}(t,r), \nonumber\\ 
	\Gamma_{\phi r}^{\phi}      &=\Gamma_{\theta r}^{\theta }=k_{9}(t,r), &
	\Gamma_{\theta \theta }^{r} &=\tfrac{\Gamma _{\phi \phi }^{r}}{\sin ^{2}\theta }=k_{10}(t,r), \nonumber\\ 
	\sin \theta \Gamma_{t\theta}^{\phi}  &=-\tfrac{\Gamma_{\phi t}^{\theta }}{\sin \theta }=k_{11}(t,r), & 
	\Gamma _{\phi \phi }^{\theta} 		 &=-\sin\theta \cos \theta \nonumber\\
	\sin \theta \Gamma _{r\theta}^{\phi} &=-\tfrac{\Gamma_{r\phi }^{\theta }}{\sin \theta }=k_{12}(t,r), & 
	\Gamma _{\theta \phi }^{\phi}        &=\Gamma_{\phi\theta }^{\phi }=\cot {\theta}\,. \label{eq:appcon}
\end{align}

The curvature tensor $R^c{}_{ab}\dot{\partial}_c = \left[ \delta _{a},\delta _{b} \right]$ of the induced nonlinear (Ehresmann) connection on $TM$ given by the coefficients $N^a{}_b=\Gamma^a_{bc}\dot{x}^c$, can be locally expressed as
\begin{equation*}
	R^a{}_{bc}=\delta _{c}N^a{}_{b}-\delta_{b}N^a{}_{c} = \partial_c N^a{}_b - N^d{}_c \dot\partial_d N^a{}_b -  \partial_b N^a{}_c + N^d{}_b \dot\partial_d N^a{}_c\,.
\end{equation*}
where:
\begin{equation}\label{eq:appcurv}
	\begin{array}{llll}
		R_{~tr}^{t}=a_{1}\dot{t}+a_{2}\dot{r} & R_{~tr}^{r}=a_{3}\dot{t}+a_{4}\dot{r} & R_{~tr}^{\theta }=a_{5}\dot{\theta} & R_{~tr}^{\varphi }=a_{5}\dot{\varphi} \\ 
		R_{~t\theta }^{t}=a_{6}\dot{\theta} & R_{~t\theta }^{r}=a_{7}\dot{\theta} & 
		R_{t\theta }^{\theta }=a_{8}\dot{t}+a_{9}\dot{r} & R_{t\theta }^{\varphi }=0
		\\ 
		R_{~t\varphi }^{t}=a_{6}\dot{\varphi}\sin ^{2}\theta  & R_{~t\varphi
		}^{r}=a_{7}\dot{\varphi}\sin ^{2}\theta  & R_{~t\varphi }^{\theta }=0 & 
		R_{t\varphi }^{\varphi }=a_{8}\dot{t}+a_{9}\dot{r} \\ 
		R_{~r\theta }^{t}=a_{10}\dot{\theta} & R_{~r\theta }^{r}=a_{11}\dot{\theta}
		& R_{~r\theta }^{\theta }=a_{12}\dot{t}+a_{13}\dot{r} & R_{~r\theta
		}^{\varphi }=0 \\ 
		R_{~r\varphi }^{t}=a_{10}\dot{\varphi}\sin ^{2}\theta  & R_{~r\varphi
		}^{r}=a_{11}\dot{\varphi}\sin ^{2}\theta  & R_{~r\varphi }^{\theta }=0 & 
		R_{~r\varphi }^{\varphi }=a_{12}\dot{t}+a_{13}\dot{r} \\ 
		R_{~\theta \varphi }^{t}=0 & R_{~\theta \varphi }^{r}=0 & R_{~\theta \varphi
		}^{\theta }=-a_{14}\dot{\varphi}\sin ^{2}\theta  & R_{~\theta \varphi
		}^{\varphi }=a_{14}\dot{\theta}\,.
	\end{array}
\end{equation}
The coefficients $a_i$ above are functions of $t$ and $r$, as follows:
\begin{align}\label{a_i}
	\begin{split}
		a_{1}  &= k_{1,r}-k_{2,t}+k_{3}k_{4}-k_{2}k_{6}\,,\\
		a_{2}  &= k_{2,r}-k_{3,t}+k_{2}^{2}+k_{3}k_{6}-k_{1}k_{3}-k_{2}k_{5}\,,  \\
		a_{3}  &= k_{4,r}-k_{6,t}+k_{1}k_{6}+k_{4}k_{5}-k_{2}k_{4}-k_{6}^{2}\,, \\
		a_{4}  &= k_{6,r}-k_{5,t}+k_{2}k_{6}-k_{3}k_{4}\,, \\
		a_{5}  &= k_{8,r}-k_{9,t}\,, \\
		a_{6}  &= -k_{7,t}+k_{7}k_{8}-k_{1}k_{7}-k_{2}k_{10}\,,\\
		a_{7}  &= -k_{10,t}+k_{8}k_{10}-k_{4}k_{7}-k_{6}k_{10}\,, \\
		a_{8}  &= -k_{8,t}+k_{1}k_{8}+k_{4}k_{9}-k_{8}^{2}\,,\\
		a_{9}  &= -k_{9,t}+k_{2}k_{8}+k_{6}k_{9}-k_{8}k_{9}  \,, \\
		a_{10} &= -k_{7,r}+k_{7}k_{9}-k_{2}k_{7}-k_{3}k_{10} \,,  \\
		a_{11} &= -k_{10,r}+k_{9}k_{10}-k_{6}k_{7}-k_{5}k_{10} \,,  \\
		a_{12} &= -k_{8,r}+k_{2}k_{8}+k_{6}k_{9}-k_{8}k_{9}\,,\\
		a_{13} &=-k_{9,r}+k_{3}k_{8}+k_{5}k_{9}-k_{9}^{2} \,, \\
		a_{14} &= 1+k_{7}k_{8}+k_{9}k_{10}  \,.
	\end{split}
\end{align}
The curvature components $R^a{}_{bcd}$ of the affine connection can be obtained by $\dot{x}$-differentiation from $R^{a}_{~cd}$, as: $R^a{}_{bcd}=\dot{\partial}_{b}R^{a}_{~cd}$; e.g., $a_1=R^{t}{}_{ttr}$ etc..

 \begin{acknowledgments}
	C.P. was funded by the cluster of excellence Quantum Frontiers funded by the Deutsche Forschungsgemeinschaft (DFG, German Research Foundation) under Germany's Excellence Strategy - EXC-2123 Quantum Frontiers - 390837967.
	S.Ch. was funded by the Transilvania Fellowships for Young Researchers 2021 grant. 
	This article is based upon work from COST Actions: CA21136 (Addressing observational tensions in cosmology with systematics and fundamental physics - CosmoVerse) and CA18108 (Quantum Gravity Phenomenology in the Multi-Messenger Approach), supported by COST (European Cooperation in Science and Technology).
 \end{acknowledgments}

\bibliographystyle{utphys}
\bibliography{SBSVac}

\end{document}